\documentclass
{amsart}

\usepackage{amssymb,amscd}
\usepackage[all,line,arc,curve,color,frame,pdf]{xy}
\usepackage{tikz} 
\usepackage{tikz-cd}
\usetikzlibrary{positioning, trees, snakes}
\usepackage[textsize=tiny]{todonotes}
\usepackage{hyperref}
\usepackage[shortlabels]{enumitem}
\usepackage{float} 
\usepackage[paper=a4paper, margin=3.5cm]{geometry}
\usepackage{cleveref}  
\usepackage{comment}
\usepackage{scalerel}
\usepackage[normalem]{ulem}
\usepackage{booktabs}
\usepackage{mathtools}
\usepackage{cite} 

\numberwithin{equation}{section}
\theoremstyle{plain}
\newtheorem{thm}{Theorem}[section]
\newtheorem{cor}[thm]{Corollary}
\newtheorem{lemma}[thm]{Lemma}

\theoremstyle{remark}
\newtheorem{rem}[thm]{Remark}
\newtheorem{ex}[thm]{Example}

\theoremstyle{definition}
\newtheorem{defi}[thm]{Definition}

\def\BBB{\mathcal{B}}

\def\NN{\mathbb{N}}

\subjclass[2020]{primary: 05B35, 52B40, 13F65 
}
\usepackage[textsize=tiny]{todonotes}

\usepackage{enumitem}

\newcommand\ignore[1]{}

\newcommand\CC{{\mathbb{C}}}

\newcommand\RR{{\mathbb{R}}}
\newcommand\QQ{{\mathbb{Q}}}

\def\operatorname#1{\mathop{\rm #1}\nolimits}

\def\rk{\operatorname{rk}}

\def\deg{\operatorname{deg}}

\newcommand{\pb}{\ar@{}[dr]|{\text{\pigpenfont J}}}

\newcommand{\xdasharrow}[2][->]{
\tikz[baseline=-\the\dimexpr\fontdimen22\textfont2\relax]{
\node[anchor=south,font=\scriptsize, inner ysep=1.5pt,outer xsep=2.2pt](x){#2};
\draw[shorten <=3.4pt,shorten >=3.4pt,dashed,#1](x.south west)--(x.south east);
}}

\begin{document}
\title{White's conjecture for matroids and inner projections}


\author[Han]{Kangjin Han}
\address{ 
School of Undergraduate Studies, Daegu-Gyeongbuk Institute of Science \& Technology (DGIST),
Daegu 42988, Republic of Korea}
\email{kjhan@dgist.ac.kr}

\author[Micha{\l}ek]{Mateusz Micha{\l}ek}
\address{
	University of Konstanz, Germany, Fachbereich Mathematik und Statistik, Fach D 197
	D-78457 Konstanz, Germany
}
\email{mateusz.michalek@uni-konstanz.de}

\author[Weigert]{Julian Weigert}
\address{Max Planck Institute for Mathematics in the Sciences/University of Leipzig, Leipzig, Germany
}
\email{julian.weigert@mis.mpg.de}

\begin{abstract}
White's conjecture predicts quadratic generators for the ideal of any matroid base polytope. We prove that White's conjecture for any matroid $M$ implies it also for any matroid $M'$, where $M$ and $M'$ differ by one basis. Our study is motivated by inner projections of algebraic varieties. 
\end{abstract}

\thanks{K.H. is supported by a National Research Foundation of Korea (NRF) grant (MSIT no. RS-2024-00414849) and DGIST Global Visiting Research Program. MM is supported by the DFG grant 467575307}
\maketitle

\section{Introduction}

Matroids are central objects in combinatorics. 
They generalize the notion of independence, appearing in various mathematical disciplines such as graph theory and linear algebra.
Matroids are well-known for having many equivalent cryptomorphic definitions. For the purpose of this article, we recall the definition using circuits, for which the reader unfamiliar with matroids should imagine cycles in a graph.

A \emph{matroid} on a finite set $E$ is a collection $\mathcal{C}\subseteq 2^E$ of nonempty subsets of $E$ called \emph{circuits} satisfying the following conditions for any distinct $c_1,c_2\in \mathcal{C}$:
\begin{itemize}
\item  $c_1\nsubseteq c_2$
\item if $e\in c_1\cap c_2$ then there exists $c\in \mathcal{C}$ such that $c\subset c_1\cup c_2\setminus\{e\}$.
\end{itemize}
For readers familiar with a different definition of a matroid, we recall that circuits are inclusion minimal dependent subsets. Bases are inclusion maximal subsets not containing any circuit. Let $\BBB(M)$ be the set of all bases of $M$.
For more information about matroids we refer to \cite{oxley2006matroid} and for a brief introduction to matroids to \cite[Chapter 13.1]{jaBernd}.

Matroids exhibit many relations to various branches of mathematics, including algebra, geometry and optimization. 
In this article we focus on White's conjecture, which is one of the central open problems in matroid theory. Given a field $k$ and a matroid $M$ on the set $E$ we consider a morphism of polynomial rings:
\begin{equation}\label{f map}
    f_M:k[y_b]_{b\in \BBB(M)}\rightarrow k[x_e]_{e\in E}~,
\end{equation}
defined by $f_M(y_b):=\prod_{e\in b} x_e$. The \emph{toric ideal $I_M$ of a matroid} $M$ is defined to be the kernel of $f_M$. The weak version of White's conjecture \cite{white1980unique} states that $I_M$ is generated by quadrics. Its strong version predicts explicit generators, namely binomials $y_{B_1}y_{B_2}-y_{B_1\cup\{e_1\}\setminus\{e_2\}}y_{B_2'\cup\{e_2\}\setminus \{e_1\}}$, i.e.~binomials corresponding to symmetric basis exchanges. White also posed a version of the conjecture for noncommutative rings, however we do not consider it in the present article.

All these conjectures remain open in general. Partial results confirm (the strong) White's conjecture e.g.~for graphic \cite{blasiak2008toric}, sparse paving  \cite{BONIN20136} and strongly base orderable \cite{lason2014toric}. Further there exist results confirming it 'up to saturation' \cite{lason2014toric} or showing equivalence of strong and weak conjecture for special classes like regular \cite{berczi2024reconfiguration} or split \cite{berczi2024exchange} matroids. A bound on the degree of Gr\"obner basis elements, in terms of the rank of a matroid, was provided in \cite{lason2021toric}.

From the point of view of algebraic geometry White's conjecture is particularly interesting for representable matroids. Indeed by the groundbreaking work \cite{gelfand1987combinatorial} we know that each (Cartan) torus orbit closure in a Grassmannian, say over $\CC$, is exactly defined by $I_M$ for a matroid $M$ representable over $\CC$. Even in this very classical, geometric case, White's conjecture remains open.

In the work \cite{berczi2024reconfiguration} Seymour's decomposition theorem for regular matroids plays a central role. From this perspective it is natural to ask which modifications of matroids preserve White's conjecture. Certain operations trivially have this property, like taking the dual or direct sum. 
This paper establishes that White’s conjecture is preserved under single-basis modifications and provides a novel proof technique of the conjecture for sparse paving matroids.
To make the statement precise we make the following definition.
\begin{defi}
    Let $M$ be a matroid on a ground set $E$ with the set of bases $\BBB{(M)}$. For a $b\in \BBB{(M)}$, let us assume that $\BBB{(M)}\setminus\{b\}$ is the set of bases of a matroid on $E$. In such a case, we denote the new matroid by $M_b$.
\end{defi}
Our main result is as follows.
\begin{thm}\label{main_thm}
    Assume that $M$ and $M_b$ are matroids. Then the strong White's conjecture holds for $M$ if and only if it holds for $M_b$.
\end{thm}
As a corollary, we obtain an easy proof of Bonin's result stating  that sparse paving matroids satisfy the strong version of White's conjecture \cite{BONIN20136} (see Corollary \ref{thm_Bonin}).  

Finally, we would like to remark that our main result  is motivated by the following geometric consideration. Theorem \ref{main_thm} is naturally related to an `inner projection' of the algebraic variety defined by the toric ideal $I_M$. Here, by an inner projection we mean a linear projection of a variety $V$ from an internal point of $V$. From (\ref{f map}), we have a commutative diagram as follows:
\begin{equation}\label{inner_diagram}
\begin{tikzcd}
f: k[y_{b'}]_{b'\in\BBB} \ar[r] & k[x_e]_{e\in E} \\
f_b: k[y_{b'}]_{b'\in\BBB\setminus \{b\}} \ar[r]\ar[hookrightarrow]{u}&  k[x_e]_{e\in E}\ar[equal]{u}
\end{tikzcd}.
\end{equation}
By the diagram (\ref{inner_diagram}) we can see that $I_{M_b}$, the toric ideal for $M_b$, is the elimination ideal $I_M \cap k[\{y_{b'}\}\setminus y_b]$. Thus, $I_{M_b}$ can be interpreted as the ideal of the projection of $V(I_M)$, the projective variety defined by $I_M$, from a point corresponding to $y_b$. 
In coordinates, the point from which we project is given by $y_b=1$ and $y_{b'}=0$ for any $b'\neq b$. It is easy to see that this point  corresponds to an \textit{internal} point of $V(I_M)$.
See \cite{han_kwak12} for general results on some algebraic and homological invariants of inner projections.

\section*{Acknowledgements}
We would like to thank Micha{\l} Laso{\'n} and Alex Fink for interesting discussions on the topic. Mateusz Micha{\l}ek thanks the Institute for Advanced Study for a great working environment and support through the Charles Simonyi Endowment.

\section{Main Results}
\begin{defi}
Let $M=(E,\BBB(M))$ be a matroid. For a basis $b\in \BBB(M)$ and an element $x\in E\setminus b$, we denote by $C_{b,x}$ the unique circuit contained in $b\cup x$. 
\end{defi}
The elements of $C_{b,x}$ other than $x$ are precisely those for which we can exchange $x$ in $b$, that is for any $y \in b$ \begin{align}
\label{fundamental_circuits_exchange}
    b\setminus y \cup x \in \BBB(M) \Leftrightarrow y \in C_{b,x}.
\end{align}

\begin{rem}
We note that if a matroid has a loop or a coloop we may delete it or contract it. The set of bases of the matroid and obtained minor can be canonically identified. The statements like White's conjecture either hold or fail for both matroids simultaneously under such identification.

For this reason, from now on we will assume that our matroid does not contain loops or coloops. 
\end{rem}

We will be using the following very special properties of matroid base polytopes.
\begin{lemma}\cite[Section 4]{gelfand1987combinatorial}\label{lem:edge=sym}
    A lattice polytope $P$ is a matroid base polytope if and only if it is a subpolytope of a cube $[0,1]^n\subset \RR^n$ and each edge is a shift of a vector $e_i-e_j$ for some $i,j\in \{1,\dots,n\}$, where $e_k$ form the standard basis of $\RR^n$. 

    In such a case edges of $P_M$ correspond exactly to symmetric basis exchanges.
\end{lemma}
\begin{defi}
    Let $M=(E,\BBB(M))$ be a matroid with matroid base polytope $P_M$. Let $b\in \BBB(M)$ and $v_b\in P_M$ be the corresponding vertex. We define $P_{M_b}$ to be the lattice polytope that is the convex hull of all vertices of $P$ distinct from $v_b$. We note that $P_{M_b}$ may, but does not have to be a matroid base polytope. 
\end{defi}

\begin{cor}\label{cor:PMb=PMcapH-}
     Let $M=(E,\BBB(M))$ be a matroid with matroid base polytope $P_M$. Let $b\in \BBB(M)$. Let $H_b$ be the hyperplane $\sum_{i\in b} x_i=(\rk M)-1$ and let $H_b^-$ be the half space defined by $\sum_{i\in b} x_i\leq (\rk M)-1$. Then:
     \begin{enumerate}
         \item Each vertex adjacent to $v_b$ belongs to $H_b$.
         \item $P_{M_b}=P_M\cap H_b^-$. 
     \end{enumerate}
     
\end{cor}
\begin{proof}
    By Lemma \ref{lem:edge=sym} each vertex adjacet to $v_b$ is given by a basis that differs from $b$ on exactly one element, thus it satisfies the defining equation of $H_b$. The first point follows.

    For the second point the inclusion $P_{M_b}\subset P_M\cap H_b^-$ is obvious. For the other inclusion we prove that every vertex of $P_M\cap H_b^-$ belongs to $P_{M_b}$. Pick such a vertex $w$. If $w\not \in H_b$ then $w$ is a vertex of $P_M$, necessarily distinct from $v_b$, and hence in $P_{M_b}$. If $w\in H_b$, then in general there are two possibilities. Either $w$ belonged to an edge $e$ of $P_M$ or was a vertex of $P_M$. In the latter case we see $w\in P_{M_b}$, thus we only have to exclude the case $w=e\cap H_b$. Indeed, in this case the edge $e$ would need to contain $v_b$. By the first point, the other vertex of $e$ would belong to $H_b$, and thus $w$ would be a vertex of $P_M$. This finishes the proof. 
\end{proof}
\begin{ex}
 Consider the uniform matroid $M=U_{2,4}$. Its base polytope is a regular 3-dimensional octahedron in $\RR^4$. Let $b\in \BBB(M)$ be any basis then we may remove $b$ from $\BBB(M)$ to get a matroid with base polytope given by a pyramid with square base. This new matroid has five bases, say $\{1,2\},\{1,3\},\{1,4\},\{2,3\},\{2,4\}$, and we removed $b=\{3,4\}$. The four bases that are not $\{1,2\}$ cannot be removed if we want to stay in the class of matroid base polytopes. Indeed if $c$ is such a basis then intersecting with the half space $H_c^{-}$ will introduce a new edge which is a diagonal in the square base of the pyramid. This edge is not in the direction of the difference of two standard basis vectors, so $P_{M_c}$ is not a matroid base polytope. On the other hand if $c=\{1,2\}$ then intersecting with $H_c^{-}$ leaves us with a square which is the base polytope of the direct sum of two copies of $U_{1,2}$.  
 \begin{figure}[h]
 \begin{tikzpicture}[thick,scale=3]
    \coordinate (A1) at (0,0);
    \coordinate (A2) at (0.6,0.2);
    \coordinate (A3) at (1,0);
    \coordinate (A4) at (0.4,-0.2);
    \coordinate (B1) at (0.5,0.5);
    \coordinate (B2) at (0.5,-0.5);

    \begin{scope}[thick,dashed,opacity=0.6]
        \draw (A1) -- (A2) -- (A3);
        \draw (B1) -- (A2) -- (B2);
    \end{scope}

    \draw[solid][line width=2pt] (A1) -- (A4) -- (B1);
    \draw[solid][line width=2pt] (A1) -- (A4) -- (B2);
    \draw[solid][line width=2pt] (A3) -- (A4) -- (B1);
    \draw[solid][line width=2pt] (A3) -- (A4) -- (B2);
    \draw[solid][line width=2pt] (B1) -- (A1) -- (B2) -- (A3) --cycle;
\end{tikzpicture}
 \begin{tikzpicture}[thick,scale=3]
    \coordinate (A1) at (0,0);
    \coordinate (A2) at (0.6,0.2);
    \coordinate (A3) at (1,0);
    \coordinate (A4) at (0.4,-0.2);
    \coordinate (B2) at (0.5,-0.5);

    \begin{scope}[thick,dashed,opacity=0.6]
        \draw (A2) -- (B2);
    \end{scope}

    \draw[solid][line width=2pt] (A1) -- (A2) -- (A3);
    \draw[solid][line width=2pt] (A1) -- (A4) -- (B2);
    \draw[solid][line width=2pt] (A3) -- (A4);
    \draw[solid][line width=2pt] (A3) -- (A4) -- (B2);
    \draw[solid][line width=2pt] (A1) -- (B2) -- (A3);
\end{tikzpicture}
 \begin{tikzpicture}[thick,scale=3]
    \coordinate (A1) at (0,0);
    \coordinate (A2) at (0.6,0.2);
    \coordinate (A3) at (1,0);
    \coordinate (A4) at (0.4,-0.2);

    \draw[solid][line width=2pt] (A1) -- (A2) -- (A3) -- (A4) -- cycle;
\end{tikzpicture}
 \begin{tikzpicture}[thick,scale=3]
    \coordinate (A2) at (0.6,0.2);
    \coordinate (A3) at (1,0);
    \coordinate (A4) at (0.4,-0.2);
    \coordinate (B2) at (0.5,-0.5);

    \begin{scope}[thick,dashed,opacity=0.6]
        \draw (A2) -- (B2);
    \end{scope}

    \draw[solid][line width=2pt] (A2) -- (A3);
    \draw[solid][line width=2pt] (A4) -- (B2);
    \draw[solid][line width=2pt] (A3) -- (A4);
    \draw[solid][line width=2pt] (A3) -- (A4) -- (B2);
    \draw[solid][line width=2pt] (B2) -- (A3);
    \draw[solid][line width=2pt][red] (A2) --(A4);
\end{tikzpicture}
\caption{The base polytope of $U_{2,4}$ and its subpolytopes}
\end{figure}
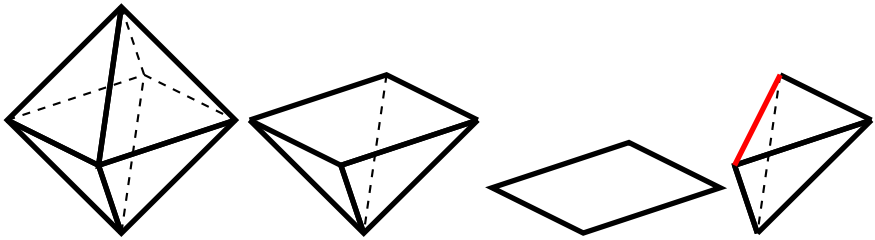
\end{ex}

In the following lemma we will show exactly under which condition the polytope $P_{M_b}$ is still the base polytope of a matroid.

\begin{lemma}\label{lem:whenwemayremoveb}
Let $M=(E,\BBB(M))$ be a matroid. Let $b\in \BBB(M)$. The set $\BBB(M)\setminus \{b\}$ is the set of bases of a matroid if and only if for any $p\in E\setminus b$ and $q\in b$ we have $b\cup p\setminus q\in \BBB(M)$. 
\end{lemma}
\begin{proof}
$\Rightarrow$: 

Claim: for any two distinct elements $x_1,x_2\in b$ and any two distinct elements $y_1,y_2\not \in b$ if $b\setminus x_1\cup y_1$ and $b\setminus x_2\cup y_2$ are bases, then so are $b\setminus x_1\cup y_2$ and $b\setminus x_2\cup y_1$.
\begin{proof}[Proof of the claim]
If the given sets are bases in $M_b$, then we may perform symmetric basis exchange among them, extending the first one by $y_2$. There are only two possibilities: either $y_1$ is removed from the first basis, or $x_2$. The first case is exactly the statement of the claim. The second one is not possible, as then the second basis would become $b$, which is not a basis in $M_b$.
\end{proof}

As $M$ does not have loops, the claim together with (\ref{fundamental_circuits_exchange}) implies that for any two elements $y_1,y_2\not\in b$ the circuits $C_{b,y_1}$ and $C_{b,y_2}$ have the same intersection with $b$. 

It remains to be proven that $C_{b,y}\cap b=b$ for any $y\not\in b$. For this, it is enough to prove that any $x\in b$ is contained in some $C_{b,y}$. Pick a basis $b'$ that does not contain $x$. Perform a symmetric basis exchange among $b,b'$ which removes $x$ from $b$. We must have some $z\in b'$ such that $b\setminus x\cup z$ is a basis. Thus $x\in C_{b,z}$ which finishes the proof of the first implication. 

$\Leftarrow$:

By Lemma \ref{lem:edge=sym} we have to prove that each edge $e$ of $P_{M_b}$ is a shifted vector $e_i-e_j$. We will prove that each edge of $P_{M_b}$ is in fact an edge of $P_M$. Pick an edge $(v,w)$ of $P_{M_b}$. We know that $v,w$ are vertices of $P_M$. By Corollary \ref{cor:PMb=PMcapH-} $P_{M_b}=P_M\cap H_b^-$. In general, a cut of  polytope with a half space, may only introduce new edges belonging to the defining hyperplane of the cut. Thus, we easily conclude if $v$ or $w$ do not belong to $H_b$. 

The remaining case is $v,w\in H_b$. We know that $v$ and $w$ correspond respectively to bases $b\setminus p\cup q$ and $b\setminus x\cup y$. If $|\{x,y,p,q\}|<4$ then the two bases are related by symmetric basis exchange and hence $(v,w)$ is an edge of $P_M$. To finish the proof we exclude the case $|\{x,y,p,q\}|=4$. Indeed in such a case the mid point of $(v,w)$ is also the mid point of a segment joining vertices corresponding to $b\setminus p\cup y$ and $b\setminus x\cup q$. Here we are precisely using our assumption that these two sets are bases of the matroid. But an edge cannot share a point with a segment joining two other vertices. This contradiction finishes the proof. 
\end{proof}


For two multisets $\{b_1,\dots,b_k\}$ and $\{b_1',\dots,b_k'\}$  of bases we write $b_1\cdots b_k=b_1'\cdots b_k'$ if the unions of elements of both multisets are equal as a multiset. This is equivalent to $\prod_{i=1}^k y_{B_i}-\prod_{i=1}^k y_{B_i'}\in I_M$. We call such binomials degree $k$ relations.

\begin{defi}
We say that a cubic relation $b_1 b_2 b_3=b_1 b_2'b_3'$ in $M$ is of type $b$ if $b_1=b$ and $b_2b_3=b_2'b_3'$ is a symmetric basis exchange.
\end{defi}

We will prove that any relation among bases of $M_b$ can be combined from quadratic relations coming from symmetric basis exchanges in $M_b$, assuming that this is possible in $M$. First we will establish this for quadratic relations. Notice in particular that this proves that when $M$ satisfies the strong version of White's conjecture and $M_b$ satisfies the weak version of White's conjecture, then $M_b$ also satisfies the strong version of White's conjecture. Afterwards we handle all cubic relations in $M_b$. Combining both results we hence show that any relation among bases in $M_b$ is a combination of symmetric basis exchanges assuming that the same holds for $M$.
For cubic relations, we will go in two steps.
In the first step we show that when we have a sequence of symmetric basis exchanges in $M$ then we can modify it such that the basis $b$ never appears in two consecutive steps of this sequence. This means that we can avoid any relations of type $b$. Then we show how to avoid a single appearance of $b$ in such a sequence.  

\begin{lemma}
    \label{quadraticcase}
    Let $M$ be a matroid and let $b \in \BBB(M)$ be such that $M_b$ is also a matroid.
    Let $b_1b_2=bb_2'=b_1'b_2''$ be a sequence of two symmetric exchanges in $M$. Then there is a sequence of symmetric exchanges in $M_b$ that starts with $b_1b_2$ and ends in $b_1'b_2''$.
\end{lemma}
\begin{proof}
    To fix notation let us assume $b_1=b\cup x\setminus y$ and $b_1'=b\cup z \setminus w$. Hence in the sequence $b_1b_2=bb_2'=b_1'b_2''$, the first symmetric exchange exchanges $x$ for $y$ and the second symmetric exchange exchanges $z$ for $w$. We clearly have $y,w \in b, x,z\notin b$, so $x\neq y,w$ and $z\neq y,w$. If we have $x=z$ or $y=w$ then the relation $b_1b_2=b_1'b_2''$ is already a symmetric exchange in $M_b$. Therefore we may assume from now on that $|\{x,y,z,w\}|=4$. \par
    Consider the two fundamental circuits $C_{b_2',y},C_{b_2'',y}$. We argue depending on whether they are equal or not. \begin{itemize}
        \item \underline{Case 1: $C_{b_2',y}=C_{b_2'',y}$.} Since we have $y\in b_1'\setminus b_2''$ and since $M_b$ is a matroid, we can find some element $p\in b_2''\setminus b_1'$ such that we can symmetrically exchange $y$ for $p$. If $b_2''\cup y\setminus p=b$ then $b_1b_2= b_1' b_2''$ is a symmetric basis exchange and we are done. 
        Otherwise, $b_2''\cup y\setminus p\in \BBB(M_b)$ and therefore $p\in C_{b_2'',y}=C_{b_2',y}$. This means that $p\in b_2'$ and also $b_2'\cup y \setminus z=b_2\cup x \setminus p \in \BBB(M)$. If this basis would equal $b$ it would mean that $b$ and $b_2'$ differ only by a one element. But then $bb_2'$ could not be equivalent to any other pair of bases. Thus we have  $b_2'\cup y \setminus z=b_2\cup x \setminus p \in \BBB(M_b)$.
        We therefore write the following sequence of symmetric exchanges.\begin{align*}
            b_1'b_2''=(b_1'\cup p \setminus y)(b_2''\cup y \setminus p) =(b_1\cup p \setminus x)(b_2\cup x \setminus p)= b_1b_2
        \end{align*}
        Here the second exchange is obtained by exchanging $z$ and $w$. Notice that by Lemma \ref{lem:whenwemayremoveb}, $b_1\cup p\setminus x$ is a basis of $M_b$ as it differs from $b$ only by exchanging one element. Further notice that it may happen that $p=x$, in that case the last step in the above sequence is trivial and can be skipped. We see that none of the bases appearing above can be equal to $b$. Therefore the entire sequence is a sequence of symmetric exchanges in $M_b$.
        \item \underline{Case 2: $C_{b_2',y}\neq C_{b_2'',y}$.} Note that both circuits of $M$ are also circuits of $M_b$. By the circuit axiom for the matroid $M_b$, we find a circuit \begin{align*}
            C\subseteq C_{b_2',y}\cup C_{b_2'',y}\setminus y \subseteq b_2'\cup b_2''= b_2''\cup z.
        \end{align*}
        Clearly the basis $b_2''$ can not contain a circuit, so we must have $z\in C_{b_2',y}\cup C_{b_2'',y}\setminus y$ and as $z\notin b_2''$ we conclude $z \in C_{b_2',y}$. This implies that $b_2'\cup y \setminus z\in \BBB(M_b)$. Also by Lemma \ref{lem:whenwemayremoveb} we know that $b\cup z \setminus y \in \BBB(M_b)$ as it differs from $b$ only in one element.  Note that $b_1\cup z\setminus x=b\cup z\setminus y=b_1'\cup w\setminus y$ and $b_2\cup x\setminus z=b_2'\cup y\setminus z=b_2''\cup y\setminus w$.
        We therefore finish the proof by writing \begin{align*}
            b_1b_2=(b\cup z\setminus y)(b_2'\cup y \setminus z)=b_1'b_2''
        \end{align*}
        where the first exchange is obtained by exchanging $z$ for $x$ and the second relation comes from exchanging $x$ for $w$. 
    \end{itemize}
\end{proof}
Note that the previous lemma implies that if strong White's conjecture holds for $M$, then weak White's conjecture holds for $M_b$.
We now go to cubic relations where we start by showing that type $b$ relations can be avoided. More precisely consider any cubic relation $\Tilde{b_1}\Tilde{b_2}\Tilde{b_3}=\Tilde{b_1'}\Tilde{b_2'}\Tilde{b_3'}$ in $M_b$ and assume that $M$ satisfies the strong version of White's conjecture. Then working in $M$ we find a sequence of symmetric exchanges, possibly involving the basis $b$ that transforms $\Tilde{b_1}\Tilde{b_2}\Tilde{b_3}$ into $\Tilde{b_1'}\Tilde{b_2'}\Tilde{b_3'}$. We will first show that we can always modify this sequence such that none of the appearing relations is a type $b$ relation. Assume we encounter a type $b$ relation $bb_1b_2=bb_1'b_2'$. This can happen in two different situations: If $b_1,b_2,b_1',b_2'\in \BBB(M_b)$ it will be enough to present a sequence of symmetric exchanges in $M$ that are not of type $b$ to go from $bb_1b_2$ to $bb_1'b_2'$. We do so in Lemma \ref{typebsimple}. However, if additionally some of the bases $b_1,b_1',b_2,b_2'$ are equal to $b$, the above procedure is impossible as any symmetric exchange starting in a triple $bbb'$ will necessarily be of type $b$. Therefore we will need to skip two steps at a time in this setting.
\begin{lemma}
\label{typebsimple}
Let $M$ be a matroid and let $b \in \BBB(M)$ be such that $M_b$ is also a matroid.
Consider a cubic relation of type $b$, say $bb_1b_2=bb_1'b_2'$. Assume further that $b_1,b_2,b_1',b_2'\neq b$ and $bb_1b_2$ is equivalent to product of some three bases, none of which is equal to $b$. Then the relation $bb_1b_2=bb_1'b_2'$ can be generated by symmetric basis exchanges in $M$ such that we avoid cubic relations of type $b$.
\end{lemma}
\begin{proof}
Since $b_1b_2=b_1'b_2'$ is a symmetric basis exchange, we may find $x\in b_1\setminus b_2, y\in b_2\setminus b_1$ such that $b_1'=b_1\cup y\setminus x$ and $b_2'=b_2\cup x \setminus y$. We consider three cases which depend on the containment of $x,y$ in $b$. \begin{itemize}
\item \underline{Case 1: $x,y\notin b$.}
Notice that we can find an element $p\in b\setminus (b_1\cup b_1')$: In fact we have $b_1\cup b_1'=b_1\cup y$ and $b_1 \neq b$, so if no such $p$ existed then $b\subseteq b_1\cup y$ and $b\neq b_1$ would force $y \in b$, contradicting our case assumption. We want to swap $p$ out of $b$ to get an obstruction to the appearance of $b$ in our relation. Consider the two fundamental circuits $C_{b_1,p}, C_{b_1',p}$. We again consider two cases. \begin{itemize}
\item \underline{Case 1.1: $C_{b_1,p}= C_{b_1',p}$.} In that case using the symmetric exchange axiom we find $q \in b_1\setminus b$ such that $b\cup q\setminus p, b_1\cup p\setminus q\in \BBB(M)$. The fact that $b_1\cup p\setminus q$ is a basis is equivalent to $q\in C_{b_1,p}$ and hence by the case assumption also $q\in C_{b_1',p}$, so $b_1'\cup p\setminus q \in \BBB(M)$ as well. We can therefore write \begin{align*}
    bb_1b_2=(b\cup q\setminus p)(b_1\cup p\setminus q)b_2=(b\cup q\setminus p)(b_1'\cup p\setminus q)b_2' =bb_1'b_2'.
\end{align*}
Note that none of the appearing bases $(b_1\cup p\setminus q),(b_1'\cup p\setminus q)$ can be equal to $b$ as they contain $x$ or $y$ respectively. Hence none of the above relations are of type $b$.\\
\item \underline{Case 1.2: $C_{b_1,p}\neq C_{b_1',p}$.} First, assume that $b_1$ does not differ from $b$ only by exchanging one element. \\
By the circuit axiom we find a circuit $C\subseteq C_{b_1,p}\cup C_{b_1',p}\setminus p\subseteq b_1\cup b_1'=b_1\cup y$. As the circuit $C$ can not be contained in the basis $b_1$ we conclude $y\in C_{b_1',p}$ and hence $b_1'\cup p\setminus y \in \BBB(M)$. Also note that $b_1'\cup p\setminus y=b_1 \cup p \setminus x$. Furthermore by Lemma \ref{lem:whenwemayremoveb} and since $x,y\notin b$ we have $b\cup x\setminus p, b\cup y\setminus p \in \BBB(M)$. We therefore have relations \begin{align*}
    bb_1b_2&=(b\cup x\setminus p)(b_1\cup p \setminus x)b_2\\&=(b\cup x\setminus p)(b_1'\cup p \setminus y)b_2= (b\cup y \setminus p)(b_1'\cup p \setminus y)b_2' =bb_1'b_2'.
\end{align*}
As we assumed that $b_1$ does not differ from $b$ only by exchanging one element, we know that $(b_1\cup p \setminus x)=(b_1'\cup p \setminus y)$ is not equal to $b$. Hence the above sequence of symmetric exchanges does not contain relations of type $b$.

The case when $b_2$ does not differ from $b$ only by exchanging one element is similar. Thus we have to exclude the case where both $b_1$ and $b_2$ differ from $b$ by exactly one element. For contradiction assume that $b_1=b\cup x\setminus a$ and $b_2=b\cup y\setminus c$. By assumptions of the lemma $b b_1 b_2=b_i b_j b_k$, where $b_i,b_j,b_k$ are some basis distinct from $b$. All of these six bases must contain $b\setminus\{a,c\}$. Apart from this, in the multiset that is the sum $a,c$ must appear each with multiplicity $2$. If none of $b_i,b_j,b_k$ contain both $a,c$, then their union as multiset contains at most three elements of this type. As it must contain four, without loss of generality we have $a,c\in b_i$. But this implies $b_i=b$ which is not possible. 
\end{itemize}
\item \underline{Case 2: $x\in b, y \notin b$.} In this case $b\cup y \setminus x \in \BBB(M)$ according to Lemma \ref{lem:whenwemayremoveb}. Hence we simply write \begin{align*}
    bb_1b_2 = (b\cup y \setminus x)b_1b_2'= bb_1'b_2'
\end{align*}
where both relations are symmetric basis exchanges that are not of type $b$.
\item \underline{Case 3: $y\in b, x \notin b$.} Exchange the roles of $x,b_1,b_1'$ with $y,b_2,b_2'$ respectively in the previous case.
\item \underline{Case 4: $x,y\in b$.} Like in Case 1.2 without loss of generality assume that $b_1$ does not differ from $b$ by exchanging only one element. By the symmetric exchange property in $M$ we can find $p\in b_1\setminus b$ such that $b_1 \cup y\setminus p, b\cup p \setminus y\in \BBB(M)$. By Lemma \ref{lem:whenwemayremoveb} we also have $b\cup p \setminus y \in \BBB(M)$ and hence we write \begin{align*}
    bb_1b_2=(b\cup p\setminus y)(b_1\cup y \setminus p)b_2= (b\cup p\setminus x)(b_1\cup y \setminus p)b_2'= bb_1'b_2'.
\end{align*}
Since we assumed that $b_1$ does not differ from $b$ by just exchanging one element, we know that $b_1\cup y\setminus p=b_1'\cup x \setminus p$ is not equal to $b$. Hence these relations are symmetric basis exchanges that are not of type $b$.
\end{itemize}
\end{proof}

The next step of our proof is to avoid the single appearances of $b$ throughout the exchange sequence in $M$. This is done in the next lemma.
\begin{lemma}
\label{removingb}
Let $M$ be a matroid and let $b\in \BBB(M)$ be such that $M_b$ is also a matroid. Consider two cubic relations in $M$, given by symmetric exchanges:
\[b_1b_2b_3=b b_2'b_3=b_1'b_2'b_3'.\]
where $b_1,b_2,b_3,b_1',b_2',b_3'\in \BBB(M_b)$. Then $b_1b_2b_3=b_1'b_2'b_3'$ is generated by quadrics in the toric ideal of $M_b$.
\end{lemma}
\begin{proof}
Let $x,y,z,w$ be such that \begin{align*}
    b_1&=b\cup x \setminus y, \text{ and hence } b_2'=b_2\cup x \setminus y,\\
    b_1'&=b\cup z \setminus w, \text{ and hence } b_3'=b_3\cup w\setminus z .
\end{align*}
Notice that $x,z\notin b$ while $y,w \in b$, hence we consider four cases. \begin{itemize}
    \item \underline{Case 1: $y\neq w$.} We furthermore destinguish between $x=z$ and $x \neq z$. \begin{itemize}
        \item \underline{Case 1.1: $x\neq z$.} By the symmetric exchange property of $M_b$ we find an element $p\in b_3\setminus b_1$ such that $b_3\cup w \setminus p, b_1\cup p\setminus w\in \BBB(M_b)$. We first assume that $p\neq y$, then $p\notin b$ and therefore by Lemma \ref{lem:whenwemayremoveb} we know that $b\cup p \setminus w \in \BBB(M_b)$ as it differs from $b$ only by one element. Hence we write the following sequence of symmetric exchanges. \begin{align*}
        b_1b_2b_3=(b_1\cup p \setminus w)b_2(b_3\cup w \setminus p)= (b\cup p \setminus w)b_2'(b_3\cup w \setminus p)=b_1'b_2'b_3'
    \end{align*}
    Here we first exchanged $p, w$, then $x, y$ and finally $z, p$. None of the appearing bases equals $b$. \par
    Now let us assume that $p=y$ in the above, then exchanging $p$ and $w$ between $b_1,b_3$ yields the relation \begin{align}
    \label{leftHalf_Case1.1}
        b_1b_2b_3=(b_1\cup  y\setminus w) b_2 (b_3 \cup w \setminus y)= (b\cup x \setminus w) b_2 (b_3 \cup w \setminus y). 
    \end{align}
        The equality here follows since $b_1=b\cup x \setminus y$. \par
        Let us now look at the triple $b_1'b_2'b_3'$. We have $x\in b_2'$ by definition, furthermore $x\notin b$ and $x\neq z$ imply $x \notin b_1'$. Hence we can symmetrically exchange $x \in b_2'\setminus b_1'$ with some element $q\in b_1'\setminus b_2'$ giving us $(b_2'\cup q \setminus x),(b_1'\cup x \setminus q)\in \BBB(M_b)$. If we have $q=z$ then $(b_1'\cup x \setminus q)=(b\cup x \setminus w)$ and therefore we complete (\ref{leftHalf_Case1.1}) to \begin{align*}
            &b_1b_2b_3=(b_1\cup  y\setminus w) b_2 (b_3 \cup w \setminus y)= (b\cup x \setminus w) b_2 (b_3 \cup w \setminus y)\\&=(b\cup x \setminus w) (b_2'\cup z \setminus x)b_3'=b_1'b_2'b_3'
        \end{align*}
        where the third equality is obtained from exchanging $y,z$ between $b_2,(b_3 \cup w \setminus y)$ and no appearing basis is equal to $b$. \par
        On the other hand if $q\neq z$ then $q\in b$ and therefore the set $b\cup x \setminus q$ differs from $b$ by one symmetric exchange, making it a basis of $M_b$ due to Lemma \ref{lem:whenwemayremoveb}. We then simply write \begin{align*}
            b_1'b_2'b_3'=(b_1'\cup x \setminus q)(b_2'\cup q \setminus x)b_3'=(b\cup x \setminus q)(b_2'\cup q \setminus x)b_3=b_1b_2b_3
        \end{align*}
        where in the second equality we exchanged $w,z$ between $(b_1'\cup x \setminus q), b_3'$ and in the last equality we exchanged $q,y$ between $b_1,b_2$. Since $q\neq w$ none of the appearing bases is equal to $b$.
        \item \underline{Case 1.2: $x =z$.} We have $y \in b_1'\setminus b_2'$, so by symmetric exchange in $M_b$ pick $r\in b_2'\setminus b_1'$ such that $b_1'\cup r \setminus y, b_2'\cup y \setminus r \in \BBB(M_b)$. If $r=w$ then we are done since in this case $b_1'\cup r \setminus y=b \cup x \setminus y=b_1$ and hence we reduced to a quadratic relation. Hence let us assume $r\neq w$. Then $r\notin b$ and hence $b\cup r \setminus y\in \BBB(M_b)$ as it differs from $b$ by just one symmetric exchange. We conclude by noting that \begin{align*}
            b_1b_2b_3=(b\cup r \setminus y)(b_2'\cup y \setminus r)b_3=(b_1'\cup r\setminus y)(b_2' \cup y \setminus r)b_3'=b_1'b_2'b_3'.
        \end{align*}
    \end{itemize}
    \item \underline{Case 2: $y=w$.} If we also have $x=z$ then $b_1=b_1'$ and hence there is nothing to prove, so we may additionally assume $x\neq z$. We use the symmetric exchange property in $M_b$ to find an element $q\in b_1\setminus b_3$ such that $b_1\cup z\setminus q, b_3\cup q \setminus z\in \BBB(M_b)$. Since it could happen that $q=x$, we consider two cases.
    \begin{itemize}
        \item \underline{Case 2.1: $x=q$.} In this case after exchanging $z, q$ between $b_1,b_3$, the new triple contains the basis $b_1\cup z\setminus x$. Since by the assumption of Case 2 we also have $y=w$, we have $b_1\cup z \setminus x=b_1'$ and therefore the resulting triple after this first symmetric exchange only differs from the triple $b_1'b_2'b_3'$ by a quadratic relation. In fact this quadratic relation is simply the symmetric exchange of $x$ and $y=w$ between $b_2$ and $b_3\cup x\setminus z$.
        \item \underline{Case 2.2: $x\neq q$.} Here we note that $b\cup z\setminus q \in \BBB(M_b)$ by Lemma \ref{lem:whenwemayremoveb} and hence write \begin{align*}
            b_1b_2b_3= (b_1\cup z \setminus q)b_2(b_3\cup q \setminus z)=(b\cup z \setminus q)b_2'(b_3\cup q \setminus z)= b_1'b_2'b_3'.
        \end{align*}
        The exchanged elements are first $z,p$ then $x,y=w$ and finally $p,y=w$. None of the appearing bases can be equal to $b$ as the first base change was performed within $M_b$ and all other bases in question can be distinguished from $b$ by checking membership of $q$ and $x$.
    \end{itemize}
\end{itemize}
\end{proof}
We combine the above results to obtain the following theorem. 
\begin{thm}
\label{maintheorem}
Let $M$ be a matroid for which the strong version of White's conjecture holds and let $b \in \BBB(M)$. Assume that $M_b$ is also a matroid, then the strong version of White's conjecture holds for $M_b$.
\end{thm}
\begin{proof}
Let us consider any relation $b_1\cdots b_k = b_1'\cdots b_k'$ in the toric ideal of $M_b$. By Lemma \ref{quadraticcase} we can assume $k\geq 3$. By the strong version of White's conjecture for $M$, let us pick bases $b_1^{(j)},\ldots,b_k^{(j)}\in \BBB(M)$, where $j=0,\ldots,n+1$ such that $b_i^{(0)}=b_i,b_i^{(n+1)}=b_i'$ and for each index $j\in 0,\ldots, n$, the two $k$-tuples $b_1^{(j)}\cdots b_k^{(j)}$ and $b_1^{(j+1)}\cdots b_k^{(j+1)}$ differ only by symmetrically exchanging two elements between two appearing bases. The proof goes by induction first on the degee $k$ and second on the number $l$ of copies of $b$ appearing in this sequence, i.e. on \begin{align*}
    l=|\{ (i,j)\in \{1,\ldots,k\}\times \{0,\ldots, n+1\} \mid b_i^{(j)}=b\}|.
\end{align*}
If $l=0$ then we are done as the sequence is already a sequence of symmetric exchanges in $M_b$ then. Otherwise let $j_0$ be the smallest index such that the $k$-tuple $b_1^{(j_0)}\cdots b_k^{(j_0)}$ contains a copy of $b$. Let $j_1$ be the largest index such that for every $j$ with $j_0\leq j \leq j_1$, the $k$-tuple $b_1^{(j)}\cdots b_k^{(j)}$ contains a copy of $b$. Since a symmetric exchange between two copies of $b$ is impossible, we conclude that the $k$-tuples at positions $j_0$ and $j_1$ of the sequence must contain exactly one copy of $b$. In particular, after removing one copy of $b$ from them we obtain a relation between two $k-1$-tuples in the toric ideal of $M_b$. We also have a sequence of symmetric exchanges in $M$ between them, simply by removing one copy of $b$ from all of the $k$-tuples between the steps $j_0$ and $j_1$. The first and last element in the sequence have degree $k-1$ and do not contain $b$, so by induction we may generate it by a chain $C$ of symmetric basis exchanges in $M_b$. Note that $C$ may have length different from $j_1-j_0$. Adding $b$ to each element of $C$, we obtain a new sequence of symmetric basis exchanges (where possibly we increased $l$) joining $b_1^{(j_0)}\cdots b_k^{(j_0)}$ and $b_1^{(j_1)}\cdots b_k^{(j_1)}$, where now $b$ appears exactly once in each multiset. To simplify notation we redefine $j_1$ as the number of steps in which $b_1^{(j_1)}\cdots b_k^{(j_1)}$ is reached, using the new chain.
\par

Next we want to establish $j_0=j_1$. If $j_1>j_0$ then every step between the $k$-tuples at positions $j_0$ and $j_1$ is obtained as a cubic type $b$ relation multiplied by a $k-3$-tuple of bases in $M_b$. We use Lemma \ref{typebsimple} to replace every such step between $j_0$ and $j_1$ by a sequence of symmetric exchanges such that no type $b$ relations appear. Again, the last step possibly increases the number of appearances of $b$ in our sequence, but many of these appearances are easier now: We know now that in the new sequence between the two $k$-tuples which were at positions $j_0$ and $j_1$ in the original sequence, every appearing $b$ is surrounded by two $k$-tuples which do not contain $b$. Hence every appearance of $b$ in that part and its two neighbouring $k$-tuples in the sequence are obtained by taking two cubic relations as in Lemma \ref{removingb} and multiplying it with a $k-3$-tuple of bases in $M_b$. Using the result of Lemma \ref{removingb} we replace each such appearance of $b$ first by a quadratic relation in $M_b$ multiplied by $k-2$ bases in $M_b$ and then using Lemma \ref{quadraticcase} we replace the quadratic relation by a sequence of symmetric exchanges in $M_b$. \par
Comparing the resulting sequence after this procedure with the original sequence, we see that we replaced the part between the steps $j_0$ and $j_1$ by a sequence generated by symmetric exchanges in $M_b$. Therefore the total number of $b$ appearing in the original sequence decreased by the amount of $b$ appearing between $j_0$ and $j_1$. We can therefore conclude by induction.
\end{proof}
As a corollary we obtain the following theorem.
\begin{thm}\label{thm:main2}
    Assume that $M$ and $M_b$ are matroids. Then the strong White's conjecture holds for $M$ if and only if it holds for $M_b$.
\end{thm}
\begin{proof}
    The implication $\Rightarrow$ is precisely the statement of Theorem \ref{maintheorem}, thus we only have to prove $\Leftarrow$. 

    Consider $b_1 \cdots b_k= b_1'\cdots b_k'$ in $M$. We prove that this relation can be generated by symmetric basis exchanges in $M$. If none of the $2k$ bases above equals $b$ the statement just follows from the symmetric basis exchanges in $M_b$. If some $b_i=b$ and some $b_j'=b$ we may remove it from both sides and proceed by induction on $k$. Thus we may assume $b_1=\cdots=b_s=b$, $b_{s+1},\cdots, b_k, b_1',\cdots,b_k'\neq b$. 

From all $k$-tuples of bases in $M$ fix one $\tilde b_1\cdots\tilde b_k$ that can be obtained from $b_1,\cdots, b_k$ by a series of symmetric basis exchanges and contains a minimal number of $b$'s. We claim that none of $\tilde b_i$ can be equal to $b$. This will conclude the proof as then $\tilde b_1\cdots\tilde b_k =    b_1'\cdots b_k'$ can be generated by symmetric basis exchanges in $M_b$ and hence in $M$.

For contradition assume $\tilde b_1=b$. We claim that each $\tilde b_i$ either equals $b$ or differs from $b$ by at most one element. Indeed, if some $\tilde b_i$ differs by two or more elements from $b$ we can perform a symmetric basis exchange on $\tilde b_1=b$ and $\tilde b_i$, contradicting the minimality of the number of $b$'s among $\tilde b_1\cdots\tilde b_k$.

Let us introduce a degree function on bases of $M$ by $\deg (b_i):=|b_i\setminus b|$. Our previous statement is equivalent to the fact that $\deg (\tilde b_i)$ is equal to zero or one for $i=1,\dots, k$. We additively extend $\deg$ to multisets of bases. We obtain $\deg(\tilde b_1\cdots\tilde b_k)<k$. However $\deg(\tilde b_1\cdots\tilde b_k)=\deg (b_1'\cdots b_k')$ as $\tilde b_1\cdots\tilde b_k = b_1'\cdots b_k'$. Moreover, as $b_i'\neq b$ for $i=1,\dots,k$ we have $\deg(b_i')\geq 1$. Thus $\deg (b_1'\cdots b_k')\geq k$. The contradiction \[k\leq \deg (b_1'\cdots b_k')=\deg(\tilde b_1\cdots\tilde b_k)<k\]
finishes the proof. 
\end{proof}

As an application of Theorem \ref{maintheorem} we prove that any sparse paving matroid satisfies the strong version of White's conjecture. This was originally proven by Bonin in \cite{BONIN20136}. 
\begin{defi}
    Let $M=(E,\BBB(M))$ be a matroid of rank $r$. $M$ is called \emph{sparse paving} if every subset of $E$ of cardinality $r$ is either a basis or a circuit-hyperplane, i.e. a circuit and a flat in $M$.
\end{defi}
Sparse paving matroids are also those matroids that can be obtained by starting with a uniform matroid and repeatedly removing bases from it such that after every step we still have a valid matroid. The following lemma is a restatement of a well-known result in matroid theory \cite{piff1971number}, \cite[Lemma 8]{bansal2015number}. We include a proof for the sake of completeness. 
\begin{lemma}\label{lem:sparsechar}
    $M=(E,\BBB(M))$ is a sparse paving matroid of rank $r$ if and only if there are cardinality $r$ subsets $b_1,\ldots,b_n\subseteq E$ such that $\BBB(M)=\BBB(U_{E,r})\setminus \{b_1,\ldots,b_n\}$ and for each index $i=1,\ldots,n$ the set $\BBB(U_{E,r})\setminus \{b_1,\ldots,b_i\}$ is the set of bases of a matroid $M_i$.
\end{lemma}
\begin{proof}
    First assume that $M$ is sparse paving and pick any enumeration $\{b_1,\ldots,b_n\}=\BBB(U_{E,r})\setminus \BBB(M)$ of the non-bases of $M$. Assume for contradiction the existence of a minimal index $i_0$ such that $\BBB(U_{E,r})\setminus\{b_1,\ldots,b_{i_0}\}$ fails to be the set of bases of a matroid. Then by Lemma \ref{lem:whenwemayremoveb}, there must be a smaller index $i_1<i_0$ such that $b_{i_0}$ and $b_{i_1}$ differ only by one symmetric exchange, say $b_{i_0}= b_{i_1}\cup x \setminus y$. Since $M$ is sparse paving, its cardinality $r$ non-bases must be circuit-hyperplanes. Then $r-1=|b_{i_0}\setminus y|=|b_{i_0}\cap b_{i_1}|$. However, $b_{i_0}\cap b_{i_1}$ is also a flat and $b_{i_0}$ would be a circuit that has only one element not in this flat, a contradiction. For a more general statement see \cite[Proposition 3.8]{stressedHyperplanes}. \par
    For the other direction we proceed by induction on the cardinality of $\BBB(U_{E,r})\setminus \BBB(M)$. As uniform matroids are clearly sparse paving, it is enough to show that if $M$ is sparse paving and $M_b$ is a matroid then $M_b$ is sparse paving. Let $S\subseteq E$ be a set of cardinality $r$ which is not a basis in $M_b$. If $S\neq b$ then $S$ is a non-basis of $M$ and hence a circuit and a flat in $M$. Clearly $S$ is still dependent in $M_b$ as it is not contained in any basis. Any proper subset of $S$ is contained in a basis of $M$ and hence in a basis of $M_b$ or in $b$. However a proper subset of $S$ that is contained in $b$ is also contained in some basis that differs by one symmetric exchange from $b$ as by Lemma \ref{lem:whenwemayremoveb} every such base change is a basis. Hence $S$ is a circuit in $M_b$. The set $S$ is also a flat in $M_b$ as for any element $a\in E$, $S\cup a$ contains a basis of $M$ and hence a basis of $M_b$ or $b$. By the same argument as above, if $S\cup a$ contains $b$ then it must also contain another basis of $M$. Hence we showed that any circuit-hyperplane of $M$ is a circuit-hyperplane of $M_b$. Now assume $S=b$, then clearly $S$ is a circuit in $M_b$. Also adding any element $a\in E$ to $S$, the new set $S\cup a$ contains a basis of $M_b$ by the same argument as above, so $S$ is a flat and therefore a circuit-hyperplane.
\end{proof}
\begin{ex}
Consider the uniform matroid $U_{3,7}$. By removing the bases \[\{1,2,3\}, \{3,4,5\}, \{1,5,6\}, \{2,5,7\}, \{1,4,7\}, \{3,6,7\}\] we obtain the known non-Fano matroid. By further removing $\{2,4,6\}$ we obtain the famous Fano matroid. These are sparse paving matroids.
    \end{ex}

\begin{cor}\label{thm_Bonin}
    Sparse paving matroids satisfy the strong version of White's conjecture.
\end{cor}
\begin{proof}
    Uniform matroids trivially satisfy the strong version of White's conjecture, so using the notation of the previous Lemma, Theorem \ref{maintheorem} inductively shows that each of the matroids $M_i$ satisfy the conjecture as well.
\end{proof}

    The proof of Lemma \ref{lem:sparsechar} shows that starting from the uniform matroid and applying main theorem \ref{thm:main2}, i.e.~we allow to remove or add one basis at a time, we exactly obtain the class of sparse paving matroids. Of course our main theorem applies not only to uniform matroids, and it would be interesting to see which classes of matroids one gets applying it to the known cases where White's conjecture is known to be true. Below we just present one example. 

    \begin{ex}
        Consider the rank four matroid $M$ on seven elements, which over $\QQ$ is represented by the following matrix. \begin{align*}
             \begin{bmatrix}
1 & 0 & 0 & 1 & 0 & 1 & 5 \\
0 & 0 & 0 & 0 & 1 & 1 & 7 \\
1 & 0 & 1 & 0 & 0 & 0 & 3 \\
1 & 1 & 0 & 0 & 0 & 0 & -3 
\end{bmatrix}  
        \end{align*} A dependent set in $M$ either has cardinality greater than four, equals $\{1,2,3,4\}$ or contains $\{4,5,6\}$. This matroid is not paving as $\{4,5,6\}$ is a circuit of cardinality smaller than rank. Further $b=\{1,2,5,7\}$ is a basis, such that $M_b$ is a matroid. This follows e.g. by Lemma \ref{lem:whenwemayremoveb} 
    \end{ex}
The following example shows when we can remove a basis from a realizable matroid. As a consequence we construct a family of realizable matroids $M$ which admit a basis $b$ such that $M_b$ is a matroid.
\begin{ex}
    Over any field $k$ consider a representable matroid of rank $r$ on $n$ elements $\{1,\ldots,n\}$ and assume that $b=\{1,\ldots,r\}$ is a basis. Then we can find a matrix $A\in k^{r\times n-r}$ such that $M$ is represented by \begin{align*}
 \left[\text{id}_{k^r} \middle| A\right] \in k^{r\times n}
    \end{align*} 
    By Lemma \ref{lem:whenwemayremoveb} the condition of $M_b$ being a matroid translates to the statement that any entry of $A$ is non-zero. When $k$ is an infinite field, then any loopless matroid $M'$ which is realizable over $k$ is also realizable by a matrix with non-zero entries as we may act on a representation matrix $A$ by a generic element of $\text{GL}(k^r)$ to remove all appearing zeroes. Hence for infinite $k$ and any rank $r$ matroid $M'$ on a ground set $E$ realizable over $k$ we may extend $M'$ by $r$ new elements $e_1,\ldots,e_r$ to get a new matroid $M$ on $E\cup\{e_1,\ldots,e_r\}$ such that $b=\{e_1,\ldots,e_r\}$ is a basis of $M$, $M\setminus \{e_1,\ldots,e_r\}=M'$ and $M_b$ is a matroid. 
\end{ex}

\bibliography{bibML}
\bibliographystyle{unsrt} 
\end{document}